\documentclass[12pt]{amsart}
\usepackage[colorlinks=true,citecolor=magenta,linkcolor=magenta, backref=page]{hyperref}
\usepackage{amssymb,verbatim,amscd,amsmath,graphicx}
\usepackage{enumerate}
\usepackage{graphicx}
\usepackage{caption}
\usepackage{subcaption}
\usepackage{pdfsync}
\usepackage{color,colortbl}
\usepackage{fullpage}
\usepackage{bbm}
\usepackage{comment}
\numberwithin{equation}{section}
\newtheorem{theorem}{Theorem}[section]
\newtheorem{lemma}[theorem]{Lemma}
\newtheorem{proposition}[theorem]{Proposition}
\newtheorem{corollary}[theorem]{Corollary}

\theoremstyle{definition}
\newtheorem{definition}[theorem]{Definition}
\newtheorem{conjecture}[theorem]{Conjecture}
\newtheorem{def-prop}[theorem]{Definition-Proposition}

\newtheorem{remark}[theorem]{Remark}
\newtheorem{example}[theorem]{Example}

\newtheorem*{acknowledgement}{Acknowledgements}

\DeclareMathOperator{\reg}{reg}

\DeclareMathOperator{\Ass}{Ass}

\renewcommand{\AA}{{\mathbb A}}
\newcommand{\CC}{{\mathbb C}}

\newcommand{\PP}{{\mathbb P}}
\newcommand{\ZZ}{{\mathbb Z}}
\newcommand{\NN}{{\mathbb N}}
\newcommand{\QQ}{{\mathbb Q}}

\newcommand{\XX}{{\mathbb X}}

\newcommand{\kk}{{\mathbbm k}}

\def\mm{{\frak m}}

\def\pp{{\frak p}}

\def\a{{\bf a}}

\def\x{{\bf x}}

\def\z{{\bf z}}

\def\ahat{\widehat{\alpha}}

\def\1{{\bf 1}}
\def\0{{\bf 0}}

\begin{document}
	
\title{Chudnovsky's Conjecture and the stable Harbourne--Huneke containment}

\author{Sankhaneel Bisui}
\address{Tulane University \\ Department of Mathematics \\
	6823 St. Charles Ave. \\ New Orleans, LA 70118, USA}
\email{sbisui@tulane.edu}

\author{Elo\'isa Grifo}
\address{University of California at Riverside \\ Department of Mathematics \\
    900 University Ave. \\ Riverside, CA 92521, USA}
\email{eloisa.grifo@ucr.edu}

\author{Huy T\`ai H\`a}
\address{Tulane University \\ Department of Mathematics \\
	6823 St. Charles Ave. \\ New Orleans, LA 70118, USA}
\email{tha@tulane.edu}

\author{Th\'ai Th\`anh Nguy$\tilde{\text{\^e}}$n}
\address{Tulane University \\ Department of Mathematics \\
	6823 St. Charles Ave. \\ New Orleans, LA 70118, USA}
\email{tnguyen11@tulane.edu}


\begin{abstract}
We investigate containment statements between symbolic and ordinary powers and bounds on the Waldschmidt constant of defining ideals of points in projective spaces.
We establish the stable Harbourne conjecture for the defining ideal of a general set of points. We also prove Chudnovsky's Conjecture and the stable version of the Harbourne--Huneke containment conjectures for a general set of sufficiently many points.
\end{abstract}

\maketitle


\section{Introduction} \label{sec.intro}

Chudnovsky's Conjecture suggests a lower bound for the answer to the following fundamental question: \emph{given a set of distinct points $\XX \subseteq \PP^N_\kk$, where $\kk$ is a field, and an integer $m \geqslant 1$, what is the least degree of a homogeneous polynomial in the homogeneous coordinate ring $\kk[\PP^N_\kk]$ that vanishes at each point in $\XX$ of order at least $m$}? For a homogeneous ideal $I \subseteq \kk[\PP^N_\kk]$, let $\alpha(I)$ be the least generating degree of $I$ and let $I^{(m)}$ denote its $m$-th symbolic power. By the Zariski--Nagata Theorem \cite{Zariski,Nagata,EisenbudHochster} (cf. \cite[Proposition 2.14]{SurveySP}), when $\kk$ is perfect, the answer is $\alpha ( I_\XX^{(m)} )$, where $I_\XX \subseteq \kk[\PP^N_\kk]$ is the defining ideal of $\XX$. When $\kk = \CC$, the following is equivalent to Chudnovsky's Conjecture \cite{Chudnovsky1981}.

\begin{conjecture}[Chudnovsky] \label{conj.Chud}
Let $I \subseteq \kk[\PP^N_\kk]$ be the defining ideal of a set of points in $\PP^N_\kk$. For any $m \in \NN$,
$$\dfrac{\alpha(I^{(m)})}{m} \geqslant \dfrac{\alpha(I)+N-1}{N}.$$
\end{conjecture}

Chudnovsky himself proved the bound for any set of points in $\mathbb{P}^2_\CC$ (see also \cite{HaHu}). Esnault and Viehweg \cite{EsnaultViehweg} showed moreover that for any set of points in $\mathbb{P}^N_\CC$,
$$\dfrac{\alpha(I^{(m)})}{m} \geqslant \dfrac{\alpha(I)+1}{N}.$$
These results extended the previous bound $N \alpha(I^{(m)}) \geqslant m \alpha(I)$ by Waldschmidt and Skoda \cite{Waldschmidt,Skoda}. Decades later, Ein, Lazarsfeld and Smith \cite{ELS} and Hochster and Huneke \cite{comparison} showed that $I^{(Nm)} \subseteq I^m$ for all $m \in \NN$. More generally, given any radical ideal $I$ of big height $h$ in a regular ring, $I^{(hm)} \subseteq I^m$ for all $m \in \NN$ (see \cite{MaSchwede} for the mixed characteristic case). In particular, this implies the Waldschmidt and Skoda degree bound. The containment result also gives more detailed information about the polynomials that vanish to order $m$ along $\XX$.

In an effort to improve the containment between symbolic and ordinary powers, Harbourne and Huneke \cite{HaHu} conjectured that the defining ideal $I \subseteq \kk[\PP^N_\kk]$ of any set of points in $\PP^N_\kk$ satisfies a stronger containment, namely, $I^{(Nm)} \subseteq \mm^{m(N-1)}I^m$ and $I^{(Nm-N+1)} \subseteq \mm^{(m-1)(N-1)}I^m,$ for all $m \geqslant 1$, 
where $\mm$ denotes the graded irrelevant ideal. 
Chudnovsky's Conjecture, in fact, follows as a corollary to the \emph{stable} version of these conjectures, meaning it is enough to study the case when $m \gg 0$. We will study stable versions of the conjectures of Harbourne and Huneke \cite[Conjectures 2.1 and 4.1]{HaHu}.

\begin{conjecture}[Stable Harbourne--Huneke Containment] \label{conj.HaHu}
	Let $I \subseteq \kk[\PP^N_\kk]$ be a homogeneous radical ideal of big height $h$. Then there exists a constant $r(I) \geqslant 1$, depending on $I$, such that for all $r \geqslant r(I)$, we have
	$$(1) \quad I^{(hr)} \subseteq \mm^{r(h-1)}I^r \qquad \textrm{ and } \qquad (2) \quad I^{(hr-h+1)} \subseteq \mm^{(r-1)(h-1)}I^r.$$
\end{conjecture}

We will sometimes shorten this to say that the condition holds for $r \gg 0$, but we emphasize that the large enough value of $r$ we need to take might depend on $I$. The same holds for any other stable containment we might consider.

The Stable Harbourne--Huneke Containment enables us to obtain lower bounds on the degrees of all symbolic powers by studying $I^{(m)}$ only for $m \gg 0$. By taking this approach, we prove Chudnovsky's Conjecture for any set of sufficiently many general points in $\PP^N_\kk$, where $\kk$ is an algebraically closed field of any characteristic.

\medskip

\noindent\textbf{Theorems \ref{thm.HaHu} and \ref{thm.Chud}.} Suppose that $\kk$ is an algebraically closed field of arbitrary characteristic, and $N \geqslant 3$. 
\begin{enumerate}
	\item If $s \geqslant 4^N$ then there exists a constant $r(s,N)$, depending only on $s$ and $N$, such that the stable containment $I^{(Nr)} \subseteq \mm^{r(N-1)}I^r$ holds when $I$ defines a general set of $s$ points in $\PP^N_\kk$ and $r \geqslant r(s,N)$. 
	\item Assume that $\text{char } \kk = 0$. If $s \geqslant {N^2+N \choose N}$ then there exists a constant $r(s,N)$, depending only on $s$ and $N$, such that the stable containment $I^{(Nr-N+1)} \subseteq \mm^{(r-1)(N-1)}I^r$ holds when $I$ defines a general set of $s$ points in $\PP^N_\kk$ and $r \geqslant r(s,N)$. 
\end{enumerate}
In particular, Chudnovsky's Conjecture holds for a general set of $s \geqslant 4^N$ points in $\PP^N_\kk$, where $\kk$ is an algebraically closed field of any characteristic.

\medskip

Previously, Chudnovsky's Conjecture had been shown for a \emph{general} set of points in $\PP^3_\kk$ \cite{Dumnicki2015}, a set of at most $N+1$ points in \emph{generic position} in $\PP^N_\kk$ \cite{Dumnicki2015}, a set of a binomial coefficient number of points forming a \emph{star configuration} \cite{BoH, GHM2013}, a set of points in $\PP^N_\kk$ lying on a quadric \cite{FMX2018}, and a \emph{very general} set of points in $\PP^N_\kk$ \cite{DTG2017, FMX2018}.

By saying that the conjecture holds for a \emph{very general} set of points in $\PP^N_\kk$, we mean that there exist infinitely many open dense subsets $U_m$, $m \in \NN$, of the Hilbert scheme of $s$ points in $\PP^N_\kk$ such that Chudnovsky's Conjecture holds for all $\XX \in \bigcap_{m=1}^\infty U_m$. We remove this infinite intersection of open dense subsets to show that there exists one open dense subset $U$ of the Hilbert scheme of $s$ points in $\PP^N_\kk$ such that Chudnovsky's Conjecture holds for all $\XX \in U$.

In a different approach to tighten the containment $I^{(hm)} \subseteq I^m$, Harbourne's Conjecture \cite[Conjecture 8.4.2]{Seshadri} asks if $I^{(hm-h+1)} \subseteq I^m$ holds for all $m \in \NN$. While this conjecture can fail for various choices of $m$ and $I$ \cite{counterexamples, HaSeFermat,ResurgenceKleinWiman,MalaraSzpond,RealsCounterexample,BenCounterexample,DrabkinSeceleanu}, the stable version of this conjecture, known as the \emph{Stable Harbourne Conjecture} \cite[Conjecture 2.1]{GrifoStable}, remains open.

\begin{conjecture}[Stable Harbourne Conjecture]\label{conj.Ha}
	Let $I \subseteq \kk[\PP^N_\kk]$ be a radical ideal of big height $h$. For all $r \gg 0$, we have
	$$I^{(hr-h+1)} \subseteq I^r.$$
\end{conjecture}

As in the previous stable conjectures we discussed, the large enough value of $r$ we need to take might depend on $I$.

Adding more evidence towards the Stable Harbourne Conjecture \cite{GrifoStable,GrifoHunekeMukundan,GHMGorenstein,YuXieStefan,SteinerConfigurations}, we prove this conjecture for any general set of points in $\PP^N_\kk$, by verifying a stronger conjecture on the \emph{resurgence} of $I$, a notion that we shall describe in more detail in Section \ref{sec.prel}.

\medskip

\noindent\textbf{Theorem \ref{thm.exp.resurgence.general} and Corollary \ref{cor.gen.Ha}.} 
Suppose that $\kk$ is an algebraically closed field of arbitrary characteristic. The defining ideal $I$ of a general set of points in $\PP^N_\kk$ has expected resurgence. In particular, $I$ satisfies the Stable Harbourne Conjecture, and for any $C$,
$$I^{(rN-C)} \subseteq I^r \textrm{ for all } r \gg 0.$$

\medskip

Our method in proving Theorems \ref{thm.exp.resurgence.general}, \ref{thm.HaHu} and, subsequently, Theorem \ref{thm.Chud} is to examine when we can obtain stable containment from one containment, and applying \emph{specialization}. Techniques of specialization were already used in \cite{FMX2018} to establish Chudnovsky's Conjecture for a very general set of points. The main obstacle they faced in going from a \emph{very general} set of points to a \emph{general} set of points is that specialization works for each $m$ to produce an open dense subset $U_m$ in the Hilbert scheme of $s$ points in $\PP^N_\kk$ where the corresponding containment holds, and to get a statement that holds for all $m$, one would need to take the intersection $\bigcap_{m=1}^\infty U_m$. To overcome this difficulty and to apply specialization techniques also to the stable containment problem, the novelty in our method is to apply specialization for only one power and then show that one appropriate containment would lead to the stable containment. For example, to show Chudnovsky's Conjecture for general points, we prove the following result:

\medskip

\noindent\textbf{Corollary \ref{cor.stableChud}.} 
Suppose that for some value $c \in \NN$, $I^{(hc-h)} \subseteq \mm^{c(h-1)}I^c$. For all $m \gg 0$,
$$I^{(hm-h)} \subseteq \mm^{m(h-1)}I^m.$$

In fact, we prove a much more general statement (see Theorem \ref{thm.14all} and the other results in Section \ref{sec.stable}). This gives a strategy to attack Chudnovsky's Conjecture, and more generally the Stable Harbourne--Huneke Containment, in other settings: rather than showing the containment holds for all $m$ large enough, it is sufficient to find \emph{one} $m$ for which the containment holds. We use this idea in Theorems \ref{thm.exp.resurgence.general} and \ref{thm.HaHu}, by explicitly proving that one containment is satisfied. This is achieved by using specialization to pass the problem to the generic points in $\PP^N_{\kk(\z)}$, where the situation is better understood.

Finally, note that Chudnovsky's inequality, Harbourne's conjecture and the Habourne--Huneke containment conjectures have all been proved for squarefree monomial ideals in \cite{CEHH2017}. For more related problems and questions, the interested reader is referred to, for instance, \cite{CHHVT2020,SurveySP,ContainmentSurvey}.

The paper is outlined as follows. In the next section we collect necessary notations and terminology, and also recall important results that will be used often later on. In Section \ref{sec.stable}, we present our first main result, Theorem \ref{thm.14all}, where the stable containment is derived from one containment. In Section \ref{sec.HH}, we showcase our next main results, Theorems \ref{thm.exp.resurgence.general} and \ref{thm.HaHu}, where we establish the stable Harbourne conjecture for a general set of points and the stable Harbourne--Huneke containment for a general set of sufficiently many points. The last section of the paper focuses on Chudnovsky's Conjecture. In Theorem \ref{thm.Chud}, we prove Chudnovsky's Conjecture for a general set of sufficiently many points. We also establish a slightly weaker version of Chudnovsky's Conjecture for a general set of arbitrary number of points; see Theorem \ref{thm.weakChud}. As a consequence of this result, we show in Corollary \ref{chud.symbolic} that the defining ideal of a scheme of fat points with a general support and of equal multiplicity (at least 2) satisfies a Chudnovsky-type inequality. 

\begin{acknowledgement}
	The authors thank Grzegorz Malara and Halszka Tutaj-Gasi\'nska for their suggestion to look at B\"or\"oczky configurations, and Alexandra Seceleanu for pointing to \cite{DHNSST2015}, all of which resulted in Example \ref{example.Bor}. We also thank Alexandra Seceleanu for kindly sharing the new version of \cite{DrabkinSeceleanu}, which lead to Example \ref{DrabkinSeceleanu}. We also thank Jack Jeffries for helpful discussions, and Justyna Szpond and Tomasz Szemberg for comments on an early draft of this paper. Finally, we thank the referee for their many helpful comments, which in particular lead to consider in what ways our bounds are uniform. This project included many Macaulay2 \cite{M2} computations. The third named author is partially supported by Louisiana Board of Regents (grant \# LEQSF(2017-19)-ENH-TR-25).
\end{acknowledgement}


\section{Preliminaries} \label{sec.prel}

In this section, we shall recall notations, terminology and results that will be used often in the paper. Throughout the paper, let $\kk$ be a field, and let $\kk[\PP^N_\kk]$ represent the homogeneous coordinate ring of $\PP^N_\kk$. Since most of the statements in this paper become rather straightforward when $N = 1$, we shall assume that $N \geqslant 2$. Our work evolves around symbolic powers of ideals, so let us start with the definition of these objects.

\begin{definition}
Let $S$ be a commutative ring and let $I \subseteq S$ be an ideal. For $m \in \NN$, the \emph{$m$-th symbolic power} of $I$ is defined to be
$$I^{(m)} = \bigcap_{\pp \in \Ass(I)} \left(I^mS_\pp \cap S\right).$$
\end{definition}

For an ideal $I$, the \emph{big height} of $I$ is defined to be the maximum height of a minimal associated prime of $I$. The following result of Johnson \cite{Johnson} is a generalization of the aforementioned containment of Ein--Lazarsfeld--Smith, Hochster--Huneke and Ma--Schwede \cite{ELS,comparison,MaSchwede}.

\begin{theorem}[Johnson \cite{Johnson}] \label{thm.Johnson}
Let $I \subseteq \kk[\PP^N_\kk]$ be a radical ideal of big height $h$. For all $k \geqslant 1$ and $a_1, \ldots, a_k \geqslant 0$, we have
$$I^{(hk+a_1+ \dots + a_k)} \subseteq I^{(a_1 + 1)} \cdots I^{(a_k + 1)}.$$
\end{theorem}

Another result that we shall make use of is the following stable containment of the second author of the present paper \cite[Theorem 2.5]{GrifoStable}.

\begin{theorem}[Grifo] \label{thm.Grifo}
Let $I \subseteq \kk[\PP^N_\kk]$ be a radical ideal of big height $h$. If $I^{(hc-h)} \subseteq I^c$ for some constant $c \geqslant 1$ then for all $r \gg 0$, we have
$$I^{(hr-h)} \subseteq I^r.$$
More concretely, this containment holds for all $r \geqslant hc$.
\end{theorem}

As mentioned before, for a homogeneous ideal $I \subseteq \kk[\PP^N_\kk]$, $\alpha(I)$ denotes its least generating degree. The \emph{Waldschmidt constant} of $I$ is defined to be
$$\ahat(I) := \lim_{m \rightarrow \infty} \dfrac{\alpha(I^{(m)})}{m}.$$
The Waldschmidt constant of an ideal $I$ is known to exist and, furthermore, we have 
$$\ahat(I) = \inf_{m \in \NN} \dfrac{\alpha(I^{(m)})}{m}.$$ 
See, for example, \cite[Lemma 2.3.1]{BoH}. Hence, Chudnovsky's Conjecture can be restated as in the following equivalent statement:

\begin{conjecture}[Chudnovsky]
Let $I \subseteq \kk[\PP^N_\kk]$ be the defining ideal of a set of points in $\PP^N_\kk$. Then,
$$\ahat(I) \geqslant \dfrac{\alpha(I)+N-1}{N}.$$
\end{conjecture}

Symbolic powers are particularly easier to understand for the defining ideals of points. Let $\XX = \{P_1, \dots, P_s\} \subseteq \PP^N_\kk$ be a set of $s \geqslant 1$ distinct points. Let $\pp_i \subseteq \kk[\PP^N_\kk]$ be the defining ideal of $P_i$ and let $I_\XX = \pp_1 \cap \dots \cap \pp_s$ be the defining ideal of $\XX$. Then, for all $m \in \NN$, it is a well-known fact that
$$I_\XX^{(m)} = \pp_1^m \cap \dots \cap \pp_s^m.$$

\begin{remark}
The set of all collections of $s$ not necessarily distinct points in $\PP^N_\kk$ is parameterized by the \emph{Chow variety} $G(1,s,N+1)$ of $0$-cycles of degree $s$ in $\PP^N_\kk$ (cf. \cite{GKZ1994}). Thus, a property $\mathcal{P}$ is said to hold for a \emph{general} set of $s$ points in $\PP^N_\kk$ if there exists an open dense subset $U \subseteq G(1,s,N+1)$ such that $\mathcal{P}$ holds for any $\XX \in U$.
\end{remark}

Let $(z_{ij})_{1 \leqslant i \leqslant s, 0 \leqslant j \leqslant N}$ be $s(N+1)$ new indeterminates. We shall use $\z$ and $\a$ to denote the collections $(z_{ij})_{1 \leqslant i \leqslant s, 0 \leqslant j \leqslant N}$ and $(a_{ij})_{1 \leqslant i \leqslant s, 0 \leqslant j \leqslant N}$, respectively. Let
$$P_i(\z) = [z_{i0}: \dots : z_{iN}] \in \PP^N_{\kk(\z)} \quad \text{ and } \quad \XX(\z) = \{P_1(\z), \dots, P_s(\z)\}.$$
The set $\XX(\z)$ is often referred to as the set of $s$ \emph{generic} points in $\PP^N_{\kk(\z)}$. For any $\a \in \AA^{s(N+1)}_\kk$, let $P_i(\a)$ and $\XX(\a)$ be obtained from $P_i(\z)$ and $\XX(\z)$, respectively, by setting $z_{ij} = a_{ij}$ for all $i,j$. There exists an open dense subset $W_0 \subseteq \AA^{s(N+1)}_\kk$ such that $\XX(\a)$ is a set of distinct points in $\PP^N_\kk$ for all $\a \in W_0$ (and all subsets of $s$ points in $\PP^N_\kk$ arise in this way). The following result allows us to focus on open dense subsets of $\AA^{s(N+1)}_\kk$ when discussing general sets of points in $\PP^N_\kk$.

\begin{lemma}[\protect{\cite[Lemma 2.3]{FMX2018}}] \label{lem.Hilbert}
Let $W \subseteq \AA^{s(N+1)}_\kk$ be an open dense subset such that a property $\mathcal{P}$ holds for $\XX(\a)$ whenever $\a \in W$. Then, the property $\mathcal{P}$ holds for a general set of $s$ points in $\PP^N_\kk$.
\end{lemma}

Another related notion about sets of points is that of being in \emph{generic position} or, equivalently, having the \emph{generic} (i.e., maximal) Hilbert function. For a set $\XX \subseteq \PP^N_\kk$ with defining ideal $I_\XX \subseteq \kk[\PP^N_\kk]$, let $H(\XX,t)$ denote its Hilbert function. That is,
$$H(\XX,t) = \dim_\kk \left(\kk[\PP^N_\kk]/I_\XX\right)_t.$$

\begin{definition} Let $\XX \subseteq \PP^N_\kk$ be a collection of $s$ points. Then, $\XX$ is said to be \emph{in generic position} (or equivalently, to have the \emph{generic} Hilbert function) if for all $t \in \NN$, we have
$$H(\XX,t) = \min\left\{s, {t+N \choose N}\right\}.$$ 
\end{definition}
The condition of being in generic position is an open condition so, particularly, the set of generic points in $\PP^N_{\kk(\z)}$ and a general set of points in $\PP^N_\kk$ are both in generic position.

One of our main techniques is specialization. We recall this construction following \cite{Krull1948}.

\begin{definition}[Krull]
Let $\x$ represent the coordinates $x_0, \dots, x_N$ of $\PP^N_\kk$. Let $\a \in \AA^{s(N+1)}$. The \emph{specialization} at $\a$ is a map $\pi_\a$ from the set of ideals in $\kk(\z)[\x]$ to the set of ideals in $\kk[\x]$, defined by
$$\pi_\a(I) := \{f(\a,\x) ~\big|~ f(\z,\x) \in I \cap \kk[\z,\x]\}.$$
\end{definition}

\begin{remark}
Let $\pp_i(\z)$ and $\pp_i(\a)$ be the defining ideals of $P_i(\z) \in \PP^N_{\kk(\z)}$ and $P_i(\a) \in \PP^N_\kk$, respectively. It follows from \cite[Satz 1]{Krull1948} that there exists an open dense subset $W \subseteq W_0 \subseteq \AA^{s(N+1)}$ such that, for all $\a \in W$ and any $1 \leqslant i \leqslant s$, we have
$$\pi_\a(\pp_i(\z)) = \pp_i(\a).$$
We shall always assume that $\a \in W$ whenever we discuss specialization in this paper.
\end{remark}

\begin{remark} \label{rmk.KrullSpecialization}
Observe that, by the definition and by \cite[Satz 2 and 3]{Krull1948} (see also \cite[Propositions 3.2 and 3.6]{NhiTrung1999}), for fixed $m, r, t \in \NN$, there exists an open dense subset $U_{m,r,t} \subseteq W$ such that for all $\a \in U_{m,r,t}$, we have
$$\pi_\a\left(I(\z)^{(m)}\right) = I(\a)^{(m)} \text{ and } \pi_\a\left({\mm}_{\z}^t I(\z)^r\right) = \mm^t I(\a)^r.$$
\end{remark}
\noindent Here, we use $\mm$ and $\mm_\z$ to denote the maximal homogeneous ideals of $\kk[\x]$ and $\kk(\z)[\x]$, respectively. Note that $\mm_\z$ is the extension of $\mm$ in $\kk(\z)[\x]$. We shall make use of this fact often.

One invariant that plays an important role in the study of the containment problem is the resurgence, introduced by Bocci and Harbourne \cite{BoH}.

\begin{definition}[Resurgence]
	The \emph{resurgence} of the ideal $I$ is given by
	$$\rho(I) = \sup \left\{ \dfrac{a}{b} ~\Big|~ I^{(a)} \not\subseteq I^b \right\}.$$
\end{definition}

The resurgence always satisfies $\rho(I) \geqslant 1$, and over a regular ring, the resurgence of a radical ideal is always at most the big height $h$. As noted in \cite[Remark 2.7]{GrifoStable}, the Stable Harbourne Conjecture follows immediately whenever $\rho(I) < h$. In that case, we say that $I$ has \emph{expected resurgence} \cite{GrifoHunekeMukundan}. Note that expected resurgence implies more than Stable Harbourne; in fact, given any integer $C$, if $I$ has expected resurgence then $I^{(hr-C)} \subseteq I^r$ for all $r \gg 0$ \cite[Remark 2.7]{GrifoStable}.

Throughout the paper, we will use some results that appear in the literature for the case where $\kk$ has characteristic $0$ that actually hold in more generality.

\begin{theorem}[\protect{\cite[Theorem 2.4]{FMX2018}}]\label{thm 2.4 fmx over any perfect field}
Let $\kk$ be a perfect field. Let $I$ be the defining ideal of a set $\XX$ of $s$ distinct points in $\PP^N_\kk$, and let $H$ be the defining ideal of $s$ generic points in $\PP^N_{\kk(\z)}$. Then $\alpha(H^{(m)}) \geqslant \alpha(I^{(m)})$ for each $m \geqslant 1$. Moreover, for each $m \geqslant 1$ there exists an open dense set $U_m \in \AA_\kk^{s(N+1)}$ where equality holds.
\end{theorem}

In \cite{FMX2018}, this is stated only in the case where $\kk$ has characteristic $0$. We include their argument here, but now adapting the necessary steps to obtain the more general result. We emphasize that this is precisely the same argument as in \cite{FMX2018}, but the key point that allows for this generalization is that when $\kk$ is a perfect field, the symbolic powers of radical ideals are given by differential powers.

\begin{proof}
	As before, we write
	$$P_i(\z) = [z_{i0}: \dots : z_{iN}] \in \PP^N_{\kk(\z)} \quad \text{ and } \quad \XX(\z) = \{P_1(\z), \dots, P_s(\z)\},$$
	so that $H$ is the ideal corresponding to the set $\{ P_1(\z), \ldots, P_s(\z)\}$. For each $\a \in \AA^{s(N+1)}$, we write $H(\a)$ for the ideal in $\kk[\PP^N_\kk]$ corresponding to $\{P_1(\a), \ldots, P_s(\a)\}$.
	
	Fix $m \geqslant 1$. First, we claim that for each $t \geqslant 1$,
	$$V_t := \lbrace \a \in \AA^{s(N+1)} \mid \alpha(H(\a)^{(m)}) \leqslant t \rbrace = \lbrace \a \in \AA^{s(N+1)} \mid \textrm{ there exists } f \in H(\a)^{(m)} \textrm{ of degree } t \rbrace$$
	is a closed subset of $\AA^{s(N+1)}$.
	
	By \cite[Proposition 2.14]{SurveySP}, the symbolic powers of any radical ideal in $\kk[\PP^N_\kk]$ are the differential powers of the given ideal. In particular, a polynomial $f$ satisfies $f \in H(\a)^{(m)}$ if and only if $D^{m-1} f(P_i(\a)) = 0$ for all $i$, where $D^{m-1}$ denotes the differential operators of order at most $m-1$. It is well known (see \cite[Th\'eor\'eme 16.11.2]{EGAIV}, and also \cite[Remark 2.7]{SurveySP}) that when $R = k[x_0, \ldots, x_N]$ for any field $k$ we have
	$$D^{m-1} = R \left< \frac{1}{\alpha_0!} \frac{\partial^{\alpha_0}}{\partial x_0^{\alpha_0}} \cdots \frac{1}{\alpha_N!} \frac{\partial^{\alpha_N}}{\partial x_N^{\alpha_N}} ~\Big|~ |\alpha| \leqslant m-1 \right>.$$
	For simplicity, we will write $D_{\alpha} = \frac{1}{\alpha_0!} \frac{\partial^{\alpha_0}}{\partial x_0^{\alpha_0}} \cdots \frac{1}{\alpha_N!} \frac{\partial^{\alpha_N}}{\partial x_N^{\alpha_N}}$. The factors $\frac{1}{\alpha_i!}$ do not represent elements in the field; $D_{\alpha}$ is a formal representation for the $k$-linear operator defined by $D_{\alpha}(C \underline{x}^\beta) = C{\beta \choose \alpha} \underline{x}^{\beta-\alpha}$ if $\beta_i \geqslant \alpha_i$ for all $i$, and otherwise $D_{\alpha}x(C \underline{x}^\beta) = 0$. We will abuse notation and write $D_{\alpha}$ for both the operator in $\kk[\PP^N_\kk]$ and for the one in $\kk(\z)[\PP^N_{\kk(\z)}]$, noting that the operator $D_{\alpha}$ over $\kk(\z)$ restricts to the one over $\kk$.
	
	Let $f \in R = \kk[\PP^N_\kk]$ be a homogeneous polynomial of degree $t$, and write $f = \sum_{|\alpha| = t} C_{\alpha} \underline{x}^{\alpha}$. By the description of symbolic powers above, $f \in H(\a)^{(m)}$ if and only if $D_\alpha f(P_i(\a)) = 0$ for all $|\alpha| \leqslant m-1$. Since $D_\alpha$ is $\kk$-linear, we can write $D_{\beta} f(P_i(\z)) = \sum_{|{\beta}| = t} C_{\alpha} D_{\beta} \z^{\alpha}_i$ and $D_{\beta} f(P_i(\a)) = \sum_{|\alpha| = t} C_{\alpha} D_{\beta} \a^{\alpha}_i$. Here $\z_i^{\alpha}$ denotes $z_{i0}^{\alpha_0} \cdots z_{iN}^{\alpha_0}$ and $\a_i^{\alpha}$ denotes $a_{i0}^{\alpha_0} \cdots a_{iN}^{\alpha_0}$. Moreover, since the $D_\beta$ are $\kk(\z)$-linear, $D_{\beta} f(P_i(\z))|_{\z = \a} = D_{\beta} f(P_i(\a))$.
	
	Let us write the system of equations $D_{\beta} f(P_i(\a)) = 0$, where $\beta$ ranges over all the tuples in $\NN^{N+1}$ with $|\beta| \leqslant m-1$, in matrix form. In order to do that, we use the deglex order in $\NN_0^{N+1}$, so $\alpha > \beta$ if and only if $|\alpha| > |\beta|$ or $|\alpha| = |\beta|$ and there exists $j$ such that $\alpha_i = \beta_i$ for $i \leqslant j$ and $\alpha_{j+1} > \beta_{j+1}$. Consider the system of equations
	$$\mathbb{B}_{m,t} \begin{bmatrix} C_{(t,0, \ldots, 0)} & \ldots & C_{\alpha} & \ldots & C_{(0, \ldots,0, t)} \end{bmatrix}= 0$$
 	where the rows of $\mathbb{B}_{m,t}$ are indexed by $1 \leqslant i \leqslant n$ and $|\beta| \leqslant m-1$, and the row corresponding to $i$ and $\beta$ is
 	$$\begin{bmatrix} D_{\beta} z_{i0}^t & \ldots & D_{\beta} \z_i^\alpha & \ldots & D_{\beta} z_{iN}^t \end{bmatrix}.$$
 	There exists a nonzero homogeneous polynomial $f \in H(\a)^{(m)}$ of degree $t$ if and only if there exists a nontrivial solution $C$ to the system of equations $[\mathbb{B}_{m,t}]_{\a} C = 0$, where 
 	$$[\mathbb{B}_{m,t}]_{\a} = \begin{bmatrix} D_{\beta} \z^\alpha|_{\z = \a}\end{bmatrix} = \begin{bmatrix} D_{\beta} \a^\alpha\end{bmatrix}.$$
 	If $\mathbb{B}_{m,t}$ has less rows than columns, there are nontrivial solutions to the homogeneous system $[\mathbb{B}_{m,t}]_{\a} C = 0$ for any $\a$, so $V_t = \AA^{s(N+1)}_\kk$ is closed in $\AA^{s(N+1)}_\kk$. Otherwise, there are nontrivial solutions $C$ to our system if and only if the rank of $[\mathbb{B}_{m,t}]_{\a}$ is strictly smaller than the number of columns. That is equivalent to requiring finitely many minors of $[\mathbb{B}_{m,t}]_{\a}$ to vanish, which is a closed condition on the values of $\a$. Therefore, $V_t$ is closed in $\AA^{s(N+1)}$.
 	
 	So we have shown that $V_t$
	is a closed subset of $\AA^{s(N+1)}$. Now fix $n = \alpha(H^{(m)})$. We claim that the closed set $V_n$ contains an open dense subset of $\AA^{s(N+1)}$, and thus must be all of $\AA^{s(N+1)}$. To see that, consider a homogeneous polynomial $0 \neq f \in H^{(m)}$ of degree $n$, and we may assume that it is an element of $\kk[\z][\x]$. By Remark \ref{rmk.KrullSpecialization}, there exists an open dense subset $U_m$ of $\AA^{s(N+1)}$ such that $\pi_\a\left(H^{(m)}\right) = H(\a)^{(m)}$ for all $\a \in U_m$. Moreover, perhaps after intersecting with another open dense set, we can also assume that the specialization of $f$ is nonzero for all $\a \in U_m$. In particular, $f(\a,\x)$ has degree $n$ for all $\a \in U_m$. This implies that $U_m \subseteq V_n$ and, therefore, $V_n = \AA^{s(N+1)}_\kk$.
	
	This shows that $\alpha(H^{(m)}) \geqslant \alpha(H(\a)^{(m)})$ for all $\a \in V_n = \AA^{s(N+1)}$. We also constructed an open dense set $U_m$ where equality holds, so this completes the proof of the theorem.
\end{proof}

\begin{theorem}[\protect{\cite[Theorem 2.7]{FMX2018}}]\label{chudnovsky generic points fmx}
	Let $\kk$ be an algebraically closed field of arbitrary characteristic. The defining ideal $I(\z)$ of $s$ generic points in $\PP^N_{\kk(\z)}$ satisfies Chudnovsky's Conjecture, meaning
	$$\alpha(I(\z)) \geqslant \frac{\alpha(I) + N - 1}{N}.$$
\end{theorem}

In \cite{FMX2018}, this is stated only in the case where $\kk$ has characteristic $0$. The only step in their proof where characteristic $0$ is used was when applying their \cite[Theorem 2.4]{FMX2018}, which as we wrote above holds more generally. The other results they use in this proof are \cite[Proposition 2.5 (a)]{FMX2018}, which has no assumptions on the characteristic of $\kk$, and \cite[Proposition 2.9]{GHM2013}, which holds over any algebraically closed field.
%


\section{From one containment to a stable containment} \label{sec.stable}

In this section, we prove our first main result, which establishes the stable containment from one containment. Our theorem is in the same spirit as that of Theorem \ref{thm.Grifo}, but the containments we now study include an appropriate power of the maximal ideal on the right hand side. This result also forms an essential step in proving Chudnovsky's Conjecture and the stable Harbourne--Huneke containment for a general set of points, which we will cover in later sections.

\begin{theorem} \label{thm.14all}
Let $\ell(x,y) = ax+by+d \in \ZZ[x,y]$ be a linear form. Let $I \subseteq \kk[\PP^N_\kk]$ be an ideal of big height $h$. Suppose that for some value $c \in \NN$, $I^{(hc-h)} \subseteq \mm^{hc-h-\ell(h,c)}I^c$. Then, for all $r \gg 0$, provided that $\alpha(I) \geqslant h^2(a+1)+hd+1$, we have
$$I^{(hr-h)} \subseteq \mm^{hr-h-\ell(h,r)}I^r.$$
\end{theorem}

\begin{proof} For any $r \in \NN$, write $r = qhc+t$, where $q \in \NN$ and $0 \leqslant t \leqslant hc-1$. Applying Theorem \ref{thm.Johnson} for $k = qh+q+t-1$, $a_1 = \dots = a_{qh} = hc-h-1$ and $a_{qh+1} = \dots = a_k = 0$, we have
\begin{align*}
I^{(hr-h)} & \subseteq \big(I^{(hc-h)}\big)^{qh}I^{q+t-1} \\
& \subseteq \big(\mm^{hc-h-\ell(h,c)}I^c\big)^{qh}I^{q+t-1} \\
& = \mm^{(hc-h-\ell(h,c))qh}I^{q-1}I^{qhc+t} \\
& \subseteq \mm^{(hc-h-\ell(h,c))qh}\mm^{\alpha(I)(q-1)}I^r.
\end{align*}
Thus, it suffices to prove that, for $\alpha(I) \geqslant h^2(a+1)+hd+1$ and $r \gg 0$,
$$(hc-h-\ell(h,c))qh+\alpha(I)(q-1) \geqslant hr-h-\ell(h,r).$$
This is equivalent to showing that
\begin{align} \label{eq.largeQ}
\alpha(I)(q-1)-q[h^2(a+1)+hd] \geqslant ht-h-ah-bt-d.
\end{align}
Note that the right hand side is bounded, so the inequality holds for $\alpha(I) \geqslant h^2(a+1)+hd+1$ and $q \gg 0$.
\end{proof}

Particular interesting consequences of Theorem \ref{thm.14all} include the following choices of $\ell(x,y)$.

\begin{corollary} \label{cor.stableChud}
Let $b \in \ZZ$. Suppose that for some value $c \in \NN$, $I^{(hc-h)} \subseteq \mm^{c(h-b)}I^c$. For all $r \gg 0$, we have
$$I^{(hr-h)} \subseteq \mm^{r(h-b)}I^r.$$
\end{corollary}

\begin{proof} 
Pick $\ell(x,y) = -x+by$. Note that, in this case, the bound $\alpha(I) \geqslant h^2(a+1)+hd+1$ in Theorem \ref{thm.14all} is $\alpha(I) \geqslant 1$, which is trivially satisfied. The assertion now follows directly from Theorem \ref{thm.14all}.
\end{proof}

\begin{remark}
	\label{rmk.largeR}
	The bound for large values of $r$ in Corollary \ref{cor.stableChud} can be taken to be independent of $I$; for instance, in the proof of Theorem \ref{thm.14all}, one can choose $q \geqslant (h-b)(hc-1)+1$ so that the inequality (\ref{eq.largeQ}) holds, and it follows that the containment hold for $r\geqslant hc[(h-b)(hc-1)+1]$.
\end{remark}

\begin{corollary} Let $b \in \ZZ$. Suppose that for some value $c \in \NN$, $I^{(hc-h)} \subseteq \mm^{(c-1)(h-b)}I^c$. For all $r \gg 0$ and $\alpha(I) \geqslant h^2-hb+1$, we have
$$I^{(hr-h)} \subseteq \mm^{(r-1)(h-b)}I^r.$$
\end{corollary}

\begin{proof} Pick $\ell(x,y) = by-b$. The assertion follows from Theorem \ref{thm.14all}.
\end{proof}

\begin{corollary} \label{cor.HaHu1}
Suppose that $\text{char } \kk = 0$. If for some value $c \in \NN$ and $\alpha(I) \geqslant h^2+1$, $I^{(hc-h)} \subseteq \mm^{hc-h-c}I^c$ then for all $r \gg 0$, we have
$$I^{(hr-h+1)} \subseteq \mm^{(r-1)(h-1)}I^r.$$
\end{corollary}

\begin{proof} 
By applying Theorem \ref{thm.14all} with $\ell(x,y) = y$, it follows that for all $r \gg 0$ and $\alpha(I) \geqslant h^2+1$, we have
$$I^{(hr-h)} \subseteq \mm^{hr-h-r}I^r.$$
Moreover, since $\text{char } \kk = 0$, we have $I^{(hr-h+1)} \subseteq \mm I^{(hr-h)}$. The assertion now follows.
\end{proof}

\begin{remark}
	\label{rmk.largeR2}
	The bound for large values of $r$ in Corollary \ref{cor.HaHu1} can be taken to be dependent of only $\alpha(I)$, not the actual ideal $I$. This again follows from (\ref{eq.largeQ}).
\end{remark}

Theorem \ref{thm.14all} allows us to establish the stable containment for special configurations of points for which the containment in low powers is known to fail.

\begin{example}[Fermat configurations]\label{example fermat}
	The ideal $I = \left( x(y^n-z^n), y(z^n-x^n), z(x^n-y^n) \right)$ in $\kk[x,y,z]$ corresponds to a Fermat configuration of $n^2+3$ points in $\mathbb{P}^2$, where $\kk$ is a field of characteristic not $2$ and containing $n \geqslant 3$ distinct $n$-th roots of unity; see \cite{JustynaFermat} for more on Fermat configurations. This was the first class found of ideals that fail $I^{(hn-h+1)} \subseteq I^n$ for some $n$ --- more precisely, these ideals fail $I^{(3)} \subseteq I^2$ \cite{counterexamples,HaSeFermat}. Using Corollary \ref{cor.stableChud} with $b=1$, we can now establish the stable Harbourne--Huneke containment for these ideals:
	\begin{itemize}
		\item $I^{(2r-1)} \subseteq \mm^{r-1}I^r$ for $r \gg 0$, and
		\item $I^{(2r)} \subseteq \mm^{r}I^r$ for $r \gg 0$.
	\end{itemize}	
	In fact, we will show a stronger containment holds, namely, $I^{(2r-2)} \subseteq \mm^r I^r$ for all $r \gg 0$. 	By \cite[Theorem 2.1]{DHNSST2015}, the resurgence of any of the Fermat ideals $I$ above is $\rho(I) = \frac{3}{2}$. This guarantees that $I^{(2c-2)} \subseteq I^c$ whenever $\frac{2c-2}{c} > \frac{3}{2}$ or, equivalently, $c \geqslant 5$. Moreover, note that $I$ is generated by elements of degree $n+1$, and for all $m \geqslant 1$  
	$$I^{(m)} = (x^n-y^n,y^n-z^n)^m \cap (x,y)^m \cap (x,z)^m \cap (y,z)^m.$$
	In particular,
	$$\alpha(I^{(10)}) \geqslant \alpha \left( (x^n-y^n,y^n-z^n)^{10} \right)= 10n.$$
	Notice that $\alpha(I)=\omega(I)=n+1$, where $\omega(I)$ is the maximal degree of a generator of $I$. Since $n \geqslant 3$, we have $10n \geqslant 6n + 12$ (which is the generating degree of $\mm^6 I^6$), and thus $I^{(10)} \subseteq \mm^6 I^6$. Hence, by Corollary \ref{cor.stableChud}, $I^{(2r-2)} \subseteq \mm^r I^r$ for all $r \gg 0$. 
\end{example}

\begin{example}[B\"or\"oczky configuration $B_{12}$] \label{example.Bor}
	Let $I$ be the defining ideal of the B\"or\"oczky configuration $B_{12}$ of 19 triple points in $\mathbb{P}^2$, as described in \cite[Figure 1]{DHNSST2015}. By \cite[Theorem 2.2]{DHNSST2015}, $\alpha(I) = \omega(I)= 5$ and $\rho(I) = 3/2$, which as before implies that $I^{(2c-2)} \subseteq I^c$ for $c \geqslant 5$. Moreover, in the proof of \cite[Theorem 2.2]{DHNSST2015}, it is shown that $\alpha(I^{(m)}) \geqslant 4m$ for all $m$, which in particular implies that $\alpha(I^{(8)}) \geqslant 32$. Since $\omega(I^5) = 5 \cdot \omega(I) = 25$, we conclude that $I^{(8)} \subseteq \mm^7 I^5$, and by Corollary \ref{cor.stableChud} this guarantees that $I^{(2r-2)} \subseteq \mm^r I^r$ for all $r \gg 0$. Once again, this establishes the stable Harbourne--Huneke containment for $I$:
	\begin{itemize}
		\item $I^{(2r-1)} \subseteq \mm^{r-1}I^r$ for $r \gg 0$, and
		\item $I^{(2r)} \subseteq \mm^{r}I^r$ for $r \gg 0$.
	\end{itemize}
\end{example}

\begin{example}\label{DrabkinSeceleanu}
Any finite group $G$ generated by pseudoreflections determines an arrangement $\mathcal{A} \subseteq \mathbb{C}^{\textrm{rank} (G)}$ of hyperplanes, where each hyperplane is fixed pointwise by one of the pseudoreflections in $G$. Given such an arrangement $\mathcal{A}$, it turns out that the symbolic powers of the radical ideal $I = J(\mathcal{A})$ determining the singular locus of $\mathcal{A}$ are very interesting. In fact, several of the counterexamples to Harbourne's Conjecture that appear in the literature, such as the Fermat \cite{counterexamples,HaSeFermat}, Klein and Wiman configurations \cite{Klein,Wiman,Seceleanu}, turn out to arise in this form. Drabkin and Seceleanu studied this class of ideals \cite{DrabkinSeceleanu}, and in particular completely classified which $I$ among these satisfy the containment $I^{(3)} \subseteq I^2$.

Fix such an ideal $I$, which has pure height $2$, and assume that $G$ is an irreducible reflection group of rank three. By \cite[Proposition 6.3]{DrabkinSeceleanu}, $I^{(2r-1)} \subseteq I^r$ for all $r \geqslant 3$. We claim that in fact $I^{(2r-2)} \subseteq \mm^r I^r$ for $r \gg 0$. To check that, we can make small adaptations to the proof of \cite[Proposition 6.3]{DrabkinSeceleanu}. First, note that the case when $G = G(m,m,3)$ leads to the Fermat configurations, and we have shown our claim in Example \ref{example fermat}. Otherwise, note that is enough to show that $\alpha(I^{(2r-2)}) \geqslant \reg(I^r) + r$ for $r \gg 0$, using \cite[Lemma 2.3.4]{BoH}. On the one hand, by the proof of \cite[Proposition 6.3]{DrabkinSeceleanu}, $\alpha(I^{(2r-2)}) \geqslant (2r-2) \alpha(I) - 2 (2r-3)$; on the other hand, $\reg(I^r) = (r+1)\alpha(I)-2$ for $r \geqslant 2$, by \cite[Theorem 2.5]{NagelSeceleanu}. We are thus done whenever 
$$(2r-2) \alpha(I) - 2 (2r-3) \geqslant (r+1)\alpha(I)-2 + r,$$
or equivalently,
$$\alpha(I) \geqslant 5 + \frac{7}{r-3}.$$
As in \cite[Proposition 6.3]{DrabkinSeceleanu}, this holds for $r \geqslant 10$ as long as $G \notin \lbrace A_3, D_3, B_3, G(3,1,3), G_{25} \rbrace$, since in that case $\alpha(I) \geqslant 6$.

Finally, Macaulay2 \cite{M2} computations show the following:
\begin{enumerate}
	\item When $G = A_3$, $I^{(8)} \subseteq \mm^5 I^5$.
	\item When $G = D_3$, $I^{(8)} \subseteq \mm^5 I^5$.
	\item When $G = B_3$, $I^{(4)} \subseteq \mm^3 I^3$.
	\item When $G = G(3,1,3)$, $I^{(4)} \subseteq \mm^3 I^3$.
	\item When $G = G_{25}$, $I^{(4)} \subseteq \mm^3 I^3$.
\end{enumerate}

By Corollary \ref{cor.stableChud}, we conclude that $I^{(2r-2)} \subseteq \mm^r I^r$ for $r \gg 0$, which gives us the stable Harbourne--Huneke containment for $I$:
	\begin{itemize}
		\item $I^{(2r-1)} \subseteq \mm^{r-1}I^r$ for $r \gg 0$, and
		\item $I^{(2r)} \subseteq \mm^{r}I^r$ for $r \gg 0$.
	\end{itemize}

Finally, we automatically obtain $I^{(2r-1)} \subseteq \mm^r I^r$ for $r \gg 0$, and as we will see in Proposition \ref{pro.HaHu}, that implies $\ahat(I) \geqslant \frac{\alpha(I)+1}{2}$.
\end{example}


\section{Stable Harbourne and Harbourne--Huneke Containment for Points} \label{sec.HH}

In this section, we establish the stable Harbourne conjecture and the stable Harbourne--Huneke containment for a general set of points. We shall begin with the stable Harbourne conjecture, Conjecture \ref{conj.Ha}. Throughout the section, we make the assumption that $\kk$ is an algebraically closed field of arbitrary characteristic.

To say that Stable Harbourne or the Stable Harbourne--Huneke Conjectures hold for $s$ general points in $\PP^N_\kk$ is to find an open dense set $U$ of the Hilbert scheme of $s$ points in $\PP^N_\kk$ such that for all $\XX \in U$, there exists $r(\XX)$ depending on $\XX$ such that the defining ideal $I = I(\XX)$ satisfies the corresponding stable containment for all $r \geqslant r(\XX)$. However, in the case of the Stable Harbourne--Huneke Conjectures, we will show something stronger: that there exists a constant $r(s,N)$ depending only on $s$ and $N$ and an open dense subset $U$ of the Hilbert scheme of $s$ points in $\PP^N_\kk$ such that all $\XX \in U$, $I = I(\XX)$ satisfies
	$$I^{(Nr)} \subseteq \mm^{r(N-1)}I^r \textrm{ and } I^{(Nr-N+1)} \subseteq \mm^{(r-1)(N-1)}I^r$$
	for all $r \geqslant r(s,N)$. For each of these stable containment conjectures, one might wonder whether there exists a uniform bound $r(R)$, depending only on the ring $R$ but not on $I$, such that all radical ideals $I$ in $R$ satisfy the corresponding stable containment for all $r \geqslant r(R)$. 
	
	While we do not know whether such a uniform bound exists, we remark on a potential approach that fails: that we cannot uniformly bound the resurgence away from $N$, or more generally the big height of $I$.


\begin{lemma}\label{lem.comtainment.gen.Exp.Resurgence}
Assume that $N \geqslant 3$. Let $I(\z)$ be the defining ideal of $s$ generic points in $\PP^N_{\kk(\z)}$. There exists a constant $r(s,N)$, depending only on $s$ and $N$, such that for all $r \geqslant r(s,N)$, we have
$$I(\z)^{(Nr-N+1)} \subseteq\mm I(\z)^r.$$	
\end{lemma}	

\begin{proof}
For simplicity of notation, we shall write $I$ for $I(\z)$ in this proof. Let $d$ be such that 
$${{ N+d-1}\choose N}  \leqslant s < {{ N+d}\choose N}.$$
Then, $\alpha(I)=d$ and $\reg(I) \leqslant d+1$, by \cite[Lemma 5.8]{MN2001} and \cite[Corollary 1.6]{GM1984}. Now, by \cite[Remark 2.3]{YuXieStefan}, we have
  $$I^{(Nr-N+1)} \subseteq I^r \textrm{ for } r \gg 0.$$
Note that throughout \cite{YuXieStefan}, $\kk$ is assumed to have characteristic $0$, but their argument carries through independently of the characteristic of $\kk$. On the other hand, by \cite[Lemma 2.3.4]{BoH}, to show that $I^{(Nr-N+1)} \subseteq I^r$ it is sufficient to prove $\alpha(I^{(Nr-N+1)} ) \geqslant r \reg(I)$. And indeed, we will prove that there exists a constant $r(s,N)$ such that for all $r \geqslant r(s,N)$,
$$ \alpha( I^{(Nr-N+1)} ) \geqslant r \reg(I) +1,$$
which then implies that $I^{(Nr-N+1)} \subseteq \mm I^r$. 
First, note that it follows from \cite[Theorem 2.7]{FMX2018}, as stated in Theorem \ref{chudnovsky generic points fmx}, that
$$\alpha(I^{(Nr-N+1)}) \geqslant (Nr-N+1) \, \frac{\alpha(I)+N-1}{N} = \frac{(Nr-N+1)(d+N-1)}{N}.$$
We claim that if $r \geqslant r(s,N) = \frac{N+(N-1)(d+N-1) }{N(N-2)}$ then 
$$(Nr-N+1) \, \frac{(d+N-1 )}{N} \geqslant r(d+1)+1 \geqslant r \reg(I)+1.$$
Suppose not. Then,
\begin{align*}
(Nr-N+1)\, \frac{(d+N-1 )}{N} &< r(d+1)+1\\
(Nr-N+1)(d+N-1) & < rN(d+1)+N\\
(N-2)rN-(N-1)(d+N-1) & < N\\
r & < \frac{N+(N-1)(d+N-1) }{N(N-2)}.\\
\end{align*} 
This contradicts our assumption on $r$, and thus $\alpha( I^{(Nr-N+1)} ) \geqslant r \reg(I) +1$, resulting in the containment, 
$$I^{(Nr-N+1)} \subseteq\mm I^r$$
for all $r \geqslant r(s,N)$.
\end{proof}	

Next we prove that an ideal defining a set of general points has expected resurgence, which we will use to prove that general sets of points satisfy the Stable Harbourne Conjecture, and which is also interesting in its own right. 

\begin{theorem}\label{thm.exp.resurgence.general}
Suppose that $N \geqslant 3$. Then the defining ideal of a general set of points in $\PP^N_\kk$ has expected resurgence.
\end{theorem}

\begin{proof}
Let $I(\z) $ be the defining ideal of $s$ generic points in $\PP^N_{\kk(\z)}$. By Lemma \ref{lem.comtainment.gen.Exp.Resurgence}, there is a constant $c$ such that,
  $$I(\z)^{(Nc-N+1)} \subseteq\mm I(\z)^c .$$ 	
It follows from \cite[Satz 2 and 3]{Krull1948} that there exists an open dense subset $U \subseteq \AA^{s(N+1)}$ such that for all $\a \in U$,
$$\pi_\a(I(\z)^{(Nc-N)}) = I(\a)^{(Nc-N)} \text{ and } \pi_\a(I(\z)^c) = I(\a)^c.$$
Thus, for all $a \in U$, we have 
  $$I(\a)^{(Nc-N+1)} \subseteq\mm I(\a)^c.$$  
Hence, by \cite[Theorem 3.3]{GrifoHunekeMukundan}, $\rho(I(\a)) < N$.
\end{proof}

We now remark that the proof of Theorem \ref{thm.exp.resurgence.general} actually says something stronger: there exists a constant $\rho(s,N)$ depending only on $s$ and $N$ such that $\rho(I(\XX)) \leqslant \rho(s,N)$ for all $\XX \in U$. 

\begin{remark}\label{upper bound resurgence}
	Suppose $I$ is an ideal of points in $\PP^N_\kk$, or more generally a radical ideal of big height $N$ whose symbolic powers are given by saturations with $\mm$, meaning $I^{(n)} = ( I^n : \mm^\infty)$ for all $n$. If $I$ satisfies $I^{(Nc-N+1)} \subseteq \mm I^c $ for some fixed $c$, the inequality $\rho(I(\a)) < N$ is obtained in \cite{GrifoHunekeMukundan} in the following way:
\begin{enumerate}[1)]
	\item As a corollary of \cite[Main Theorem]{Sw1}, there exists a constant $l$, possibly depending on $I$, such that $\mm^{ln} I^{(n)} \subseteq I^n$ for all $n \geqslant 1$.
	\item There exists a constant $k$ such that $\overline{I^{n+k}} \subseteq I^n$ for all $n \geqslant 1$. While the proof of \cite[Theorem 2.6]{GrifoHunekeMukundan} says that this constant $k$ may depend on $I$, \cite[Theorem 4.13]{HunekeUniform} states that $k$ can be taken to be independent of $I$. This is known as the Uniform Brian\c{c}on-Skoda Theorem. In fact, $k$ can be taken to be $N$, by \cite[Remark 4.14]{HunekeUniform}
	\item In \cite[Theorem 3.2]{GrifoHunekeMukundan}, it is shown that the assumption implies that $I^{(cn)} \subseteq \mm^{\lfloor\frac{n}{c}\rfloor} I^n$ for all $n \geqslant c$, which then allows us to use \cite[Theorem 2.6]{GrifoHunekeMukundan}.
	\item In \cite[Theorem 2.6]{GrifoHunekeMukundan}, constants $t < r$ depending on $c$, $l$, $k$, and $N$ are constructed such that $\rho(I)<\frac{tN}{r}$ (cf. \cite[Lemma 2.5]{GrifoHunekeMukundan}).
\end{enumerate}

Ultimately, the bound obtained for $\rho(I)$ depends on $k$, which depends only on $N$, on $c$, which by Lemma \ref{lem.comtainment.gen.Exp.Resurgence} depends only on $N$ and $s$, and on $l$. From what we have said so far, only $l$ might depend on $I$, and the remaining constants depend only on $s$ and $N$. Note, however, that when $I$ defines a set of points in $\kk[\PP^N_\kk]$, \cite[Theorem 1.1]{GeramitaGimiglianoPitteloud} gives $\reg(I^n) \leqslant n \reg(I)$. Since the saturation degree of $I^n$ is at most $\reg(I^n)$, this says in particular that $I^{(n)}$ and $I^n$ coincide in degrees above $n \reg(I)$. Thus, if we take $l = \reg(I)$, we have $\mm^{ln} I^{(n)} \subseteq (I^{(n)})_{\geqslant \reg(I) n} = (I^n)_{\geqslant \reg(I) n} \subseteq I^n$. Finally, we conclude that our bound on $\rho(I)$ depends on $s$, $N$, and $\reg(I)$.

On the other hand, the regularity specializes by, for example, \cite[Theorem 4.2]{NhiTrung1999}. Therefore, the open dense subset $U$ in the proof of Theorem \ref{thm.exp.resurgence.general} can be further chosen such that for all $\a \in U$, $\reg(I(\a)) = \reg(I(\z))$ is independent of the points. Hence, Theorem \ref{thm.exp.resurgence.general} says in fact that there is an open dense set $U \subseteq G(1,s,N+1)$ where the resurgence is uniformly bounded away from $N$, meaning there exists a constant $\rho(s,N)$ depending only on $s$ and $N$ such that $\rho(I(\a)) \leqslant \rho(s,N)$ for all $\a \in U$. 

One might wonder if more generally we can bound the resurgence uniformly away from the big height $h$ for all radical ideals $I$. The answer is no: Bocci and Harbourne give an example (constructed by Ein) in \cite[Lemma 2.4.3(b)]{BoH} of a sequence of ideals $I_n$ of big height $h$ such that $\rho(I_n) \xrightarrow{n \rightarrow \infty} h$. In particular, these examples include sets of points $\XX_n$ in $\PP^{N}_\kk$ such that $\rho(I(\XX)) \xrightarrow{n \rightarrow \infty} N$ \cite[Lemma 2.4.3(a)]{BoH}. However, the number of points $s_n$ in $\XX_n$ also grows with $n$. We also note that these examples still satisfy Stable Harbourne --- in fact, they satisfy the original conjecture of Harbourne's \cite[Example 8.4.8]{Seshadri}.
\end{remark}

As a consequence, we conclude that a general set of points satisfies the Stable Harbourne Conjecture and, particularly, we obtain the following stable containment:

\begin{corollary}\label{cor.gen.Ha}
	Assume that $N \geqslant 3$. Let $I$ denote either the ideal of $s$ generic points in $\PP^N_{\kk(\z)}$ or the ideal of $s$ general points in $\PP^N$. For any given integer $C$ and all $r \gg 0$, we have
	$$I^{(rN-C)} \subseteq I^r.$$
\end{corollary}
	         
\begin{proof}
In the case of generic points, Lemma \ref{lem.comtainment.gen.Exp.Resurgence} gives $c$ such that $I^{(cN-N+1)} \subseteq \mm I^c$, which by \cite[Theorem 3.3]{GrifoHunekeMukundan} implies $\rho(I) < N$. In the case of general points, Theorem \ref{thm.exp.resurgence.general} says that $\rho(I) < N$. The stable containment $I^{(rN-C)} \subseteq I^r$ follows immediately (cf. \cite[Remark 2.7]{GrifoStable}).
\end{proof}

Again, in the case of general points we showed something a bit stronger.

\begin{remark}
	In fact, as detailed in \cite[Remark 2.7]{GrifoStable}, for each $C$, the containment $I^{(rN-C)} \subseteq I^r$ holds for all $r \geqslant \frac{C}{N-\rho(I)}$. As we discussed in Remark \ref{upper bound resurgence}, there is an open dense set $U \subseteq G(1,s,N+1)$ where the resurgence is uniformly bounded away from $N$. So if $\rho(s,N)$ is that uniform bound, every $\XX \in U$ satisfies $I^{(rN-C)} \subseteq I^r$ for all $r \geqslant \frac{C}{N-\rho(s,N)}$. Once we fix $s$ and $N$, this bound depends only on the constant $C$. 
\end{remark}

We will now continue to work towards establishing the stable Harbourne--Huneke containment in Conjecture \ref{conj.HaHu}. We will need a few lemmas, where the underlying ideas come from those used in \cite[Lemma 3 and Theorem 4]{DTG2017} (see also \cite[Lemma 3.1]{MSS2018}).

\begin{lemma} \label{lem.count}
The inequality
	$$k^N \leqslant {Nk-N-1 \choose N}$$
holds in the following cases:
\begin{enumerate}
	\item $k \geqslant 5$ and $N \geqslant 3$,
	\item $k \geqslant 4$ and $N \geqslant 4$, and 
	\item $k \geqslant 3$ and $N \geqslant 9$.
\end{enumerate}
\end{lemma}

\begin{proof} 
(1) The assertion is equivalent to
$$(kN-N-1) \dots (kN-2N) \geqslant k^N N!,$$
i.e.,
$$\left(N - \dfrac{N+1}{k}\right) \dots \left(N - \dfrac{2N}{k}\right) \geqslant N!.$$
Since the left hand side increases with $k$, to prove (1) it suffices to establish the inequality when $k = 5$. Equivalently, we need to show that ${4N-1 \choose N} \geqslant 5^N$.

Notice also that for $N \geqslant 3$, we have
$$\dfrac{{4(N+1)-1 \choose N+1}}{{4N-1 \choose N}} = \dfrac{(4N+3)(4N+2)(4N+1)(4N)}{(N+1)(3N+2)(3N+1)(3N)} \geqslant \dfrac{(4N+3)(4N)}{(N+1)(3N)} \geqslant 5.$$

Thus, it is enough to show that ${4N-1 \choose N} \geqslant 5^N$ for $N = 3$, which is simply $165 = {11 \choose 3} \geqslant 125 = 5^3$. 

(2) Similarly, it is enough to show that ${3N-1 \choose N} \geqslant 4^N$ for $k \geqslant 4$ and $N \geqslant 4$. This is indeed true since ${11 \choose 4} >4^4$ and 
	$$\dfrac{{3(N+1)-1 \choose (N+1)}}{{3N-1 \choose N}} = \dfrac{(3N+2)(3N+1)(3N)}{(N+1)(2N+1)(2N)} = \dfrac{3(3N+2)(3N+1)}{2(N+1)(2N+1)} \geqslant 4.$$

(3) Again, by the same line of arguments, it suffices to show that ${2N-1 \choose N} \geqslant 3^N$ for $k \geqslant 3$ and $N\geqslant 9$. This holds since ${17 \choose 9} > 3^9$ and 
  	$$\dfrac{{2(N+1)-1 \choose (N+1)}}{{2N-1 \choose N}}=\dfrac{(2N+1)(2N)}{(N+1)N}=\dfrac{2(2N+1)}{N+1}\geqslant 3.$$ 
\end{proof}

\begin{lemma} \label{lem.genHaHu}
	Let $I(\z)$ be the defining ideal of $s$ generic points in $\PP^N_{\kk(\z)}$. Suppose also that one of the following holds:
	\begin{enumerate}
		\item $N \geqslant 3$ and $s \geqslant 4^N$,
		\item $N \geqslant 4$ and $s \geqslant 3^N$, or 
		\item $N \geqslant 9$ and $s \geqslant 2^N$.
	\end{enumerate}
For $r \gg 0$, we have
	$$I(\z)^{(Nr-N)} \subseteq \mm_\z^{r(N-1)}I(\z)^r.$$
\end{lemma}

\begin{proof} 
For simplicity of notation, we shall write $I$ for $I(\z)$ in this proof. Let $k$ and $d$ be integers such that $(k-1)^N \leqslant s < k^N$ and ${d+N-1 \choose N} < s \leqslant {d+N \choose N}$.

By \cite[Theorem 2]{DTG2017}, the Waldschmidt constant of a very general set of $s$ points in $\PP^N_\kk$ is bounded below by $\lfloor \sqrt[N]{s}\rfloor$. We remark here that even though \cite[Theorem 2]{DTG2017} assumed $\text{char } \kk = 0$, its proof carries through for any infinite field $\kk$. This, by \cite[Theorem 2.4]{FMX2018}, implies that $\ahat(I) \geqslant \lfloor \sqrt[N]{s}\rfloor = k-1$. Furthermore, as in Lemma \ref{lem.comtainment.gen.Exp.Resurgence}, by \cite[Lemma 5.8]{MN2001} and \cite[Corollary 1.6]{GM1984}, we have $\reg(I) = d+1$.

By the same proof as the one given in \cite[Theorem 4]{DTG2017}, we get
\begin{align}
N(k-1) \geqslant N-1+(d+1) = N+d.\label{eq.genHaHu100}
\end{align}
We claim that (\ref{eq.genHaHu100}) is a strict inequality. Indeed, if $N(k-1) = N+d$ then $d = Nk-2N$. This implies that ${(Nk-2N)+N-1 \choose N} < s < k^N$. That is, ${Nk-N-1 \choose N} < k^N$, a contradiction to Lemma \ref{lem.count}. Thus, $N (k-1) > N+d$. Particularly, we get
$$N \, \ahat(I) > N-1 + \reg(I).$$
Thus, for $r \gg 0$, we have
$$N\ahat(I) \geqslant \dfrac{r}{r-1}(N-1+\reg(I)).$$
It follows that, for $r \gg 0$,
\begin{align}
\alpha(I^{(Nr-N)}) \geqslant (Nr-N)\ahat(I) = (r-1)N \,\ahat(I) \geqslant r(N-1 + \reg(I)). \label{eq.genHaHu1}
\end{align}

Now, by \cite[Lemma 2.3.4]{BoH}, we have $I^{(Nr-N)} \subseteq I^r$ for $r \gg 0$, since $\alpha(I^{(Nr-N)}) \geqslant r \reg(I)$ for $r \gg 0$. This, together with (\ref{eq.genHaHu1}), implies that $I^{(Nr-N)} \subseteq \mm_\z^{r(N-1)}I^r$ for $r \gg 0$.
\end{proof}

The stable Harbourne--Huneke containment for a general set of points in $\PP^2_\kk$ was already proved in \cite[Proposition 3.10]{HaHu}. We shall now establish a stronger version of the stable Harbourne--Huneke containment for a general set of sufficiently many points in $\PP^N_\kk$ when $N \geqslant 3$.

\begin{theorem} \label{thm.HaHu}
Suppose that $N \geqslant 3$.
\begin{enumerate}
\item If $s \geqslant 4^N$ then there exists a constant $r(s,N)$, depending only on $s$ and $N$, such that the stable containment $I^{(Nr)} \subseteq \mm^{r(N-1)}I^r$ holds when $I$ defines a general set of $s$ points in $\PP^N_\kk$ and $r \geqslant r(s,N)$. 
\item Assume that $\text{char } \kk = 0$. If $s \geqslant {N^2+N \choose N}$ then there exists a constant $r(s,N)$, depending only on $s$ and $N$, such that the stable containment $I^{(Nr-N+1)} \subseteq \mm^{(r-1)(N-1)}I^r$ holds when $I$ defines a general set of $s$ points in $\PP^N_\kk$ and $r \geqslant r(s,N)$.
\end{enumerate}
\end{theorem}

\begin{proof} Let $I(\z)$ be the defining ideal of the set of $s \geqslant 4^N$ generic points in $\PP^N_\kk$. By Lemma \ref{lem.genHaHu}, there exists a constant $c \in \NN$ such that
$$I(\z)^{(Nc-N)} \subseteq \mm_\z^{c(N-1)}I(\z)^c.$$

By \cite[Satz 2 and 3]{Krull1948}, there exists an open dense subset $U \subseteq \AA^{s(N+1)}$ such that for all $\a \in U$,
$$\pi_\a(I(\z))^{(Nc-N)} = I(\a)^{(Nc-N)}, \quad \pi_\a(I(\z)^c) = I(\a)^c \quad \text{ and } \quad \pi_\a(\mm_\z^{c(N-1)}) = \mm^{c(N-1)}.$$
Thus, for all $\a \in U$, we have
\begin{align}
I(\a)^{(Nc-N)} \subseteq \mm^{c(N-1)}I(\a)^c. \label{eq.HaHu10}
\end{align}
Note that we can pick $U$ such that for all $\a \in U$, we also have $\alpha(I(\a)) = \alpha(I(\z))$.

Applying Corollary \ref{cor.stableChud} and Remark \ref{rmk.largeR} with $b = 1$ to (\ref{eq.HaHu10}), we get that, for $\a \in U$ and $r \gg 0$ (independent of $I(\a)$),
$$I(\a)^{(Nr-N)} \subseteq \mm^{r(N-1)}I(\a)^r.$$
This proves (1).

Now, suppose that $s \geqslant {N^2+N \choose N}$. This also gives that $s \geqslant 4^N $ since for $N\geqslant 3$ we have ${N^2+N \choose N}\geqslant 4^N$. By \cite[Proposition 6]{GO1982}, $\alpha(I(\a))$ is the least integer $d$ such that ${d+N \choose N} > s$. This implies that $\alpha(I(\a)) \geqslant N^2+1$. Furthermore, it follows from (\ref{eq.HaHu10}) that, for all $\a \in U$,
$$I(\a)^{(Nc-N)} \subseteq \mm^{cN-c-N}I(\a)^c.$$
Therefore, since $\alpha(I(\a)) \geqslant N^2+1$, Corollary \ref{cor.HaHu1} and Remark \ref{rmk.largeR2} apply, noting that $\alpha(I(\a)) = \alpha(I(\z))$ is independent of $I(\a)$, to give
$$I(\a)^{(Nr-N+1)} \subseteq \mm^{(r-1)(N-1)}I^r$$
for all $\a \in U$ and $r \gg 0$ (independent of $I(\a)$). This proves (2) and, hence, the theorem.
\end{proof}

\begin{remark} \label{rem.improvegeneralcont}
By Lemma \ref{lem.genHaHu} and the same proof as that of Theorem \ref{thm.HaHu}, the stable Harbourne--Huneke containment $I^{(Nr)} \subseteq \mm^{r(N-1)}I^r$, for $r \gg 0$, is true for $s\geqslant 3^N$ general points when $N\geqslant 4$ and for $s\geqslant 2^N$ general points when $N\geqslant 9$.
\end{remark}

\begin{remark} To prove Theorem \ref{thm.HaHu}, we in fact establish a stronger containment, namely,
\begin{align}
I^{(Nr-N)} \subseteq \mm^{r(N-1)}I^r \text{ for } r \gg 0, \label{eq.strongcontainment} 
\end{align}
when $I$ is the defining ideal of a general set of (sufficiently many) points. The condition that we have a \emph{general} set of points is necessary. The containment (\ref{eq.strongcontainment}) is not true for an arbitrary set of points. For example, if $I$ is the defining ideal of ${N+d} \choose N $ points forming a \emph{star configuration} (see, for example, \cite{GHM2013} for more details on star configurations), then by \cite[Corollary 4.6]{GHM2013}, we have $ \alpha(I^{(Nr-N)} ) = (r-1)(N+d) <  \alpha( \mm^{r(N-1)}I^r ) = r(N+d)$ for $r \gg 0$, and so the containment (\ref{eq.strongcontainment}) fails. Note that the stable Harbourne--Huneke Conjectures still hold for a star configuration \cite[Corollary 3.9 and Corollary 4.7]{HaHu}.
\end{remark}


\section{Chudnovsky's Conjecture} \label{sec.Chud}

In this last section of the paper, we prove Chudnovsky's Conjecture for a general set of at least $4^N$ points in $\PP^N_\kk$. For an arbitrary number of points, we also show a weaker version of Chudnovsky's Conjecture. We again make an umbrella assumption, throughout the section, that $\kk$ is an algebraically closed field of arbitrary characteristic.

Chudnovsky's conjecture is known to hold for any set of points in $\PP^2_\kk$; see, for instance, \cite{Chudnovsky1981, HaHu}. We shall focus on the case where $N \geqslant 3$.

\begin{theorem} \label{thm.Chud}
	Suppose that $N \geqslant 3$ and $s \geqslant 4^N$. Chudnovsky's Conjecture holds for a general set of $s$ points in $\PP^N_\kk$.
\end{theorem}

\begin{proof} Let $I$ be the defining ideal of a general set of $s$ points in $\PP^N_\kk$. By Theorem \ref{thm.HaHu}, for $r \gg 0$, we have
	$$I^{(Nr)} \subseteq \mm^{r(N-1)}I^r.$$
This implies that $\alpha(I^{(Nr)}) \geqslant r(N-1)+ r\alpha(I)$. Thus,
$$\dfrac{\alpha(I^{(Nr)})}{Nr} \geqslant \dfrac{\alpha(I)+N-1}{N}.$$
By letting $r \rightarrow \infty$, we get
$$\ahat(I) \geqslant \dfrac{\alpha(I)+N-1}{N}.$$
The conjecture is proved, noting that $\ahat(I) \leqslant \dfrac{\alpha(I^{(m)})}{m}$ for all $m \in \NN$.
\end{proof}

\begin{remark}
Following Remark \ref{rem.improvegeneralcont}, Chudnovsky's Conjecture holds for $s\geqslant 3^N$ general points when $N\geqslant 4$ and for $s\geqslant 2^N$ general points when $N \geqslant 9$.
\end{remark}

Observe that Theorem \ref{thm.Chud} follows directly from the validity of Lemma \ref{lem.genHaHu}. In fact, a slightly weaker statement than Lemma \ref{lem.genHaHu} would already be enough to derive Theorem \ref{thm.Chud}.

\begin{proposition} \label{pro.HaHu}
	Let $I$ be any ideal of big height $h$. Suppose that for some constant $c \in \NN$, we have $I^{(hc-h+1)} \subseteq \mm^{c(h-1)}I^c$. Then,
	$$\ahat(I) \geqslant \dfrac{\alpha(I)+h-1}{h}.$$
\end{proposition}

\begin{proof} 
We claim that for all $t \geqslant 1$,
\begin{align}
I^{(hct)} \subseteq \mm^{(h-1)ct} I^{ct}.\label{eq.pro1}
\end{align}
This implies that
$$\alpha \left( I^{(hct)} \right) \geqslant \alpha \left( \mm^{(h-1)ct} I^{ct} \right) = ct (h-1) + \alpha(I) ct.$$
Therefore,
$$\ahat(I) = \lim_{m \rightarrow \infty} \frac{\alpha \left( I^{(m)} \right)}{m} = \lim_{t \rightarrow \infty} \frac{\alpha \left( I^{(hct)} \right)}{hct} \geqslant \frac{\alpha(I)+h-1}{h}.$$

Now we show (\ref{eq.pro1}):
\begin{align*}
I^{(hct)} & = I^{(ht + t (hc-h))} & \textrm{}\\
& \subseteq \left( I^{(hc-h+1)} \right)^t & \textrm{by Theorem \ref{thm.Johnson}}\\
& \subseteq \left( \mm^{(h-1)c} I^c \right)^t & \textrm{by the hypothesis}\\
& = \mm^{(h-1)ct} I^{ct}.
\end{align*}
The result is proved.
\end{proof}

\begin{example}[Fermat configurations revisited]
	As we saw in Example \ref{example fermat}, the ideals $I = \left( x(y^n-z^n), y(z^n-x^n), z(x^n-y^n) \right)$ in $\kk[x,y,z]$ corresponding to a Fermat configuration of $n^2+3$ points in $\mathbb{P}^2$ satisfies $I^{(2r-2)} \subseteq I^r$ for all $r \gg 0$. Similarly to what we did in Example \ref{example fermat}, one can show that $I^{(5)} \subseteq \mm^3 I^3$, which by Proposition \ref{pro.HaHu} implies that $\ahat(I) \geqslant \frac{n+1}{2}$. In fact, by \cite[Theorem 2.1]{DHNSST2015}, $\ahat(I) = n$.
\end{example}

\begin{example} 
Let $I$ be the defining ideal in $R = \mathbb{C}[x,y,z]$ of the curve $(t^3,t^4,t^5)$, i.e., the kernel of the map $x \mapsto t^4$, $y \mapsto t^4$  and $z \mapsto t^5$, which is a prime of height $2$. Macaulay2 \cite{M2} computations show that $I^{(3)} \subseteq \mm^2 I^2$, and as in the proof of Proposition \ref{pro.HaHu}, this implies that $I$ satisfies a Chudnovsky-like bound.

It is natural to ask if such a bound holds when $I$ is the defining ideal of the curve $(t^a,t^b,t^c)$, where $a, b, c$ take any value. The symbolic powers of the ideals in this class are of considerable interest, in particular because the symbolic Rees algebra of $I$ can fail to be noetherian -- although it is noetherian for $(t^3,t^4,t^5)$ above. And yet, these Chudnovsky-like bounds also hold in cases where the symbolic Rees algebra of $I$ is not noetherian. For example, when $I$ defines $(t^{25}, t^{72}, t^{29})$, which has a non-noetherian symbolic Rees algebra by \cite{NonnoetherianSymb}, Macaulay2 \cite{M2} computations also give $I^{(3)} \subseteq \mm^2 I^2$.
\end{example}

We shall now prove a slightly weaker version of Chudnovsky's Conjecture holds for any general set of arbitrary number of points in $\PP^N_\kk$. We first show that such a statement holds for any number of points if it holds for a binomial coefficient number of points.

\begin{lemma} \label{lem.binomial}
Let $f: \NN \rightarrow \QQ$ be any numerical function. If the inequality
$$\dfrac{\alpha(I^{(m)})}{m} \geqslant f(\alpha(I))$$
holds for the defining ideal of a general set of ${d+N \choose N}$ points in $\PP^N_\kk$, for any $d \geqslant 0$, then the inequality holds for the defining ideal of a general set of arbitrary number of points in $\PP^N_\kk$.
\end{lemma}

\begin{proof} 
The proof goes exactly as in \cite[Proposition 2.5(a)]{FMX2018}, replacing ``set of generic points in $\PP^N_{\kk(\z)}$'' by ``general set of points in $\PP^N_\kk$''. Note that both the set of generic points in $\PP^N_{\kk(\z)}$ and a general set of points in $\PP^N_\kk$ have the maximal (\emph{generic}) Hilbert function.
\end{proof}

The next lemma establishes one containment for a binomial coefficient number of generic points in $\PP^N_{\kk(\z)}$.

\begin{lemma} \label{lem.genChud}
	Let $I(\z)$ be the defining ideal of a set of $s = {d+N \choose N}$ generic points in $\PP^N_{\kk(\z)}$, for $N \geqslant 3$. There exists a constant $c \in \NN$ such that
	$$I(\z)^{(Nc-N)} \subseteq \mm_\z^{c(N-2)}I(\z)^c.$$
\end{lemma}

\begin{proof} 
The proof mimics that of Lemma \ref{lem.genHaHu}. By Corollary \ref{cor.gen.Ha}, for $r \gg 0$ we have
	$$I(\z)^{(Nr-N)} \subseteq I(\z)^r.$$
It suffices to show now that for $r \gg 0$, we have $$\alpha(I(\z)^{(Nr-N)}) \geqslant r(N-2+\reg(I(\z))).$$

Indeed, since $s = {d+N \choose N}$, again as in Lemma \ref{lem.comtainment.gen.Exp.Resurgence}, by \cite[Lemma 5.8]{MN2001} and \cite[Corollary 1.6]{GM1984}, we have $\alpha(I(\z)) = \reg(I(\z)) = d+1$. Thus,
$$r(N-2+\reg(I(\z))) = r(N+d-1).$$
On the other hand, by \cite[Theorem 2.7]{FMX2018}, for all $r$ we have
$$\alpha(I(\z)^{(Nr-N)}) \geqslant \dfrac{\alpha(I(\z))+N-1}{N}(Nr-N) = \dfrac{d+N}{N}(Nr-N) = (d+N)(r-1).$$

Now, for $r \gg 0$, we have 
$$\dfrac{r-1}{r} \geqslant \dfrac{d+N-1}{d+N},$$
or equivalently, $(d+N)(r-1) \geqslant r(d+N-1)$. This shows that $\alpha(I(\z)^{(Nr-N)}) \geqslant r(d+N-1)$, and we are done.
\end{proof}

We are now ready to show our general lower bound.

\begin{theorem} \label{thm.weakChud}
	Let $I$ be the defining ideal of a general set of points in $\PP^N_\kk$, for $N \geqslant 3$. For all $m \in \NN$, we have
	$$\dfrac{\alpha(I^{(m)})}{m} \geqslant \dfrac{\alpha(I)+N-2}{N}.$$
\end{theorem}

\begin{proof} By Lemma \ref{lem.binomial}, it suffices to prove the assertion for $s = {d+N \choose N}$ a binomial coefficient number.

By Lemma \ref{lem.genChud}, there exists a constant $c \in \NN$ such that
	$$I(\z)^{(Nc-N)} \subseteq \mm_\z^{c(N-2)}I(\z)^c.$$
It follows from \cite[Satz 2 and 3]{Krull1948} again that there exists an open dense subset $U \subseteq \AA^{s(N+1)}$ such that for all $\a \in U$, we have
$$\pi_\a(I(\z)^{(Nc-N)}) = I(\a)^{(Nc-N)},  \pi_\a(\mm_\z^{c(N-2)}) = \mm^{c(N-2)} \text{ and } \pi_\a(I(\z)^c) = I(\a)^c.$$
Now, applying Corollary \ref{cor.stableChud} with $b = 2$, we have, for $\a \in U$ and $r \gg 0$,
$$I(\a)^{(Nr)} \subseteq I(\a)^{(Nr-N)} \subseteq \mm^{r(N-2)}I(\a)^r.$$
This implies that, for $\a \in U$ and $r \gg 0$,
$$\dfrac{\alpha(I(\a)^{(Nr)})}{Nr} \geqslant \dfrac{\alpha(I(\a))+N-2}{N}.$$
By letting $r \rightarrow \infty$, we get $\ahat(I(\a)) \geqslant \dfrac{\alpha(I(\a))+N-2}{N}$ for all $\a \in U$, and the assertion follows.
\end{proof}

As a direct consequence of Theorem \ref{thm.weakChud}, we can show that the defining ideal of a fat point scheme over a general support and with equal multiplicity also satisfies Chudnovsky's Conjecture.

\begin{corollary} \label{chud.symbolic}
Let $J$ be the defining ideal of a fat point scheme $\XX = tP_1 + \dots + tP_s$ in $\PP^N_\kk$ with a general support, for some $t \in \NN$. If $t \geqslant 2$ then Chudnovsky's Conjecture holds for $J$, i.e.,
$$\ahat(J) \geqslant \dfrac{\alpha(J)+N-1}{N}.$$
\end{corollary}

\begin{proof} Let $I$ be the defining ideal of $\{P_1, \dots, P_s\}$. By definition, $J = I^{(t)}$.
By \cite[Lemma 3.6]{FMX2018}, for $t \geqslant \max \left\{\dfrac{N-1}{N\epsilon}, m_0\right\}$, we have
$$\ahat(J) \geqslant \dfrac{\alpha(J)+N-1}{N},$$ 
where $\epsilon > 0$ is a constant satisfying $\ahat(I)= \dfrac{\alpha(I)}{N}+\epsilon$ and $m_0$ is the integer such that $(I^{(m)})^{(t)}=I^{(mt)}$ for all $t$ and $m\geqslant m_0$. Since $I$ is the defining ideal of points, we have $m_0=1$. Furthermore, by Theorem \ref{thm.weakChud}, $\epsilon \geqslant \dfrac{N-2}{N}$ (and $\epsilon \geqslant \dfrac{N-1}{N}$ if $N = 2$; see \cite{Chudnovsky1981}). Therefore,
$\max \left\{ \dfrac{N-1}{N \epsilon}, m_0 \right\} \leqslant \dfrac{(N-1)}{N-2} \leqslant 2. $
We get the desired result.
\end{proof}

\bibliographystyle{alpha}
\bibliography{References}

\end{document}